\documentclass[smallextended]{svjour3}       
\smartqed  
\usepackage{graphicx}
\usepackage{mathptmx}      
\usepackage{amsmath, amsfonts, amssymb}
\usepackage{enumerate}
\begin{document}

\title{$\tilde{o}$rder-norm continuous operators and $\tilde{o}$rder weakly compact operators}


\author{Sajjad Ghanizadeh Zare, Kazem Haghnejad Azar, Mina Matin, Somayeh Hazrati}
\authorrunning{ \MakeLowercase{Sajjad Ghanizadeh Zare, Kazem Haghnejad Azar, Mina Matin, Somayeh Hazrati.}} 

\institute{S. Ghanizadeh  Zare\at
	Department  of  Mathematics  and  Applications, Faculty of   Sciences, University of Mohaghegh Ardabili, Ardabil, Iran.\\
	\email{s.ghanizadeh@uma.ac.ir}
	\and
	K. Haghnejad Azar\at
              Department  of  Mathematics  and  Applications, Faculty of   Sciences, University of Mohaghegh Ardabili, Ardabil, Iran. \\
              \email{haghnejad@uma.ac.ir}  
          \and
          M. Matin\at
          Department  of  Mathematics  and  Applications, Faculty of   Sciences, University of Mohaghegh Ardabili, Ardabil, Iran. \\
          \email{minamatin1368@yahoo.com}
          \and
S. Hazrati\at          
 Department of Mathematics and Applications, Faculty of Sciences, University of Mohaghegh Ardabili, Ardabil, Iran. \\
              \email{s.hazrati@uma.ac.ir  } }

\date{Received: date / Accepted: date}
\maketitle
\begin{abstract}
Let $E$ be a sublattice of a vector lattice $F$.
 A continuous operator $T$ from the vector lattice $E$ into a normed vector space $X$ is said to be $\tilde{o}$rder-norm continuous whenever $x_\alpha\xrightarrow{Fo}0$ implies $Tx_\alpha\xrightarrow{\Vert.\Vert}0$ for each $(x_\alpha)_\alpha\subseteq E$.  
Our mean from the  convergence $ x_\alpha\stackrel{Fo} {\longrightarrow} x $  is that   there exists another net $ \left(y_\alpha\right) $  in $F $ with the same index set satisfying 
$ y_\alpha\downarrow 0 $ in $F$ and $ \vert x_\alpha - x \vert \leq y_\alpha $ for all indexes $ \alpha $.
In  this paper, we will study some properties of this new class of operators and its relationships with some known classifications of operators. 
We also define the new class of operators that  named $\tilde{o}$rder weakly compact operators. A continuous operator $T: E \rightarrow X $ is said to be
$\tilde{o}$rder weakly compact, if $ T(A) $ in $X$ is a relatively weakly compact set for each $Fo$-bounded $A\subseteq E$.
In this manuscript, we study some
properties of this class of operators and its relationships with $\tilde{o}$rder-norm continuous operators.
\keywords {Vector lattice\and   property  $(F)$ \and $\tilde{o}$-convergence \and order-to-norm continuous  operator \and $\tilde{o}$rder-norm continuous operator  \and $\tilde{o}$rder weakly compact.}
\subclass{47B65 \and 46B40 \and 46B42}
\end{abstract}

\section{Introduction and Preliminaries}
Our motivation in writing this article is to communicate and expand the concepts  which have been introduced in the articles \cite{8} and \cite{p1}.
 In these papers some concepts  such as $\tilde{o}$-convergence, the  property  $(F)$ and order-to-norm continuous operators have been introduced, studied and 
 authors have been investigated some of their properties and the relationships with other lattice properties.
 We introduce the new classes of operators as the set of  $\tilde{o}$rder-norm continuous operators and we study some of their properties and their relationships with others knowns operators.

 To state our results, we need to fix some notations and recall some definitions.   
A net $(x_{\alpha})_{\alpha \in A}$ in a vector lattice $ E $ is said to be  order convergent to $x\in E$ if there is a net $(y_{\beta})_{\beta \in B} $ in $ E $ such that $ y_{\beta} \downarrow 0 $ and for every $ \beta \in B$, there exists $\alpha_{0} \in A$ such that $ | x_{\alpha} - x |\leq y_{\beta}$ whenever $ \alpha \geq \alpha_{0}$. in short, we will denote this convergence by $ x_{\alpha} \xrightarrow{o} x $ and write that $ x_{\alpha} $ is $o$-convergent to $x$. 
A net $ (x_{\alpha})_{\alpha}$ in vector lattice $ E $ is unbounded order convergent to $ x \in E $ if $ | x_{\alpha} - x | \wedge u \xrightarrow{o} 0$ for all $ u \in E^{+} $. We denote this convergence by $ x_{\alpha} \xrightarrow{uo}x $ and write that $ x_{\alpha} $ $uo$-convergent to $ x $.
It is clear that for order bounded nets, $uo$-convergence is equivalent to $o$-convergence. 
 	A net $(x_{\alpha})\subseteq E$ is said to be $\tilde{o}$rder convergent (in short, $\tilde{o}$-convergent) to $x$ if there is a net $(y_\beta)\subseteq F$, possibly  over a
 	different index set, such that
 	$ y_{\beta} \downarrow 0 $ in $F$ and for every $ \beta$,
 	there exists $\alpha_{0}$ such that
 		$ | x_{\alpha} - x |\leq y_{\beta}$ whenever $ \alpha \geq \alpha_{0}$.
 	We denote this convergence by $ x_{\alpha} \xrightarrow{Fo} x $ and write 
 	that $ (x_{\alpha}) $ is $\tilde{o}$-convergent to $x$.
 	 It is clear that if $E$ is regular in $F$ and $x_\alpha\xrightarrow{o}x$ in $E$, then $x_\alpha\xrightarrow{Fo}x$. The converse is not true in general. For example,
 $c_0$ is a sublattice of $\ell^\infty$ and $(e_n)\subseteq c_0$. $e_n\xrightarrow{\ell^\infty o}0$ in $c_0$, but it is not order convergent to $0$  in $c_0$. A subset $A$ of $E$ is said to be $F$-order bounded (in short, $Fo$-bounded), if there exist $x, y \in F$ that $A\subseteq [x,y]$.  A vector lattice $E$ is said to have the  property  $(F)$, if $A\subseteq E$ is order bounded whenever $A$ is $Fo$-bounded. (see \cite{8}).
 A Banach lattice $E$ is said to be an $AM$-space if  we have $\|x+y\|= max \{\|x\|, \|y\|\}$ for each $x,y\in E$ such that $|x|\wedge |y|=0$.
 A Banach lattice $E$ is said to be an $AL$-space if we have $\|x+y\|= \|x\|+ \|y\|$  for each $x,y\in E$ such that $|x|\wedge |y|=0$.
A Banach lattice $E$ is said to be $KB$-space whenever each increasing norm bounded sequence of $E^+$ is norm convergent.
Let $E$ and $G$ be vector spaces. We will denote $L(E,G)$ by the collections of operators from $E$ into $G$.
 $L_b(E,G)$ is the all of order bounded operators in this manuscript.
An operator $T$ from a Banach space $X$ into a Banach space $Y$ is  weakly compact if $\overline{{T(B _ X)}}$  is  weakly compact where $B _ X$ is the closed unit ball of $X$. 
 A continuous operator $T$ from Banach lattice $E$ into Banach space $X$ is called $M$-weakly compact if $\lim \Vert Tx_n\Vert=0$ holds for every norm bounded disjoint sequence $(x_n)_n$ of $E$. An operator $T:E\to F$ is said to preserve disjointness whenever $x\perp y$ in $E$ implies $Tx\perp Ty$. 
A subset $A$ of a vector lattice $E$ is called $b$-order bounded in $E$ if it is order bounded in $E^{\sim\sim}$.
If each $b$-order bounded subset of $E$ is order bounded in $E$, then $E$ is said to have the property ($b$).
 Jalili, Haghnejad and Moghimi characterized $L_{o\tau}(E,G)$ and $L^\sigma_{o\tau}(E,G)$ spaces in \cite{p1}.
An operator $T$ from a vector lattice $E$ into topological vector space $G$ is said to be order-to-topology continuous whenever $x_\alpha\xrightarrow{o}0$ implies $Tx_\alpha\xrightarrow{\tau}0$ for each $(x_\alpha)_\alpha\subseteq E$.
 For each sequence $(x_n)\subseteq E$, if $x_n\xrightarrow{o}0$ implies $Tx_n\xrightarrow{\tau}0$, then  $T$ is called  $\sigma$-order-to-topology continuous operator. 
The collection of all order-to-topology continuous operators will be denoted by $L_{o\tau}(E,G)$;
the subscript $o\tau$ is justified by the fact that the order-to-topology continuous operators;
that is, 
\begin{equation*}
L_{o\tau}(E,G)=\{T\in L(E,G):~T~\text{is order-to-topology continuous }\}.
\end{equation*}
Similarly, $L^\sigma_{o\tau}(E,G)$ represents the collection of all $\sigma$-order-to-topology continuous operators, that is, 
\begin{equation*}
L^\sigma_{o\tau}(E,G)=\{T\in L(E,G):~T~\text{is} ~\sigma-\text{order-to-topology continuous }\}.
\end{equation*}
For a normed space $G$,  $L_{on}(E,G)$  is collection of order-to-norm topology  continuous operators.

Let $E$, $G$ be two normed vector lattices. Recall that from \cite{12o}, a continuous operator $T:E\to G$ is said to be $\sigma$-$uon$-continuous, if each norm bounded $uo$-null sequence $(x_n)\subseteq E$ implies $Tx_n\xrightarrow{\|.\|}0$. Remember that 	an operator $T$ from Banach lattice $E$ into Banach space $X$ is a $wun$-Dunford-Pettis whenever $x_n \xrightarrow{wun}0$ in $E$ implies $Tx_n \xrightarrow{\|.\|}0$ in $X$ for each sequence $(x_n)\subseteq E$ (See \cite{44} for more information).

Recall that a Banach lattice $E$ is said to have the
property $(P)$ if there exists a positive contractive projection $P : E^{**} \to E$ where $E$
is identified with a sublattice of its topological bidual $E^{**}$.

In a Banach lattice $E$, a subset $A$ is said to be almost order bounded if for any $\epsilon>0$ there exists $u\in E^+$ such that $A\subseteq [-u,u]+\epsilon B_E$ ($B_E$ is the closed unit ball of $E$). One should observe the following useful fact, which can be easily verified using Riesz decomposition Theorem, that $A\subseteq [-u,u] + \epsilon B_E$ if and only if $\sup_{x\in A}\|(|x|-u)^+\| = \sup_{x\in A} \| |x| -|x|\wedge u\| \leq \epsilon$. By Theorems 4.9 and 3.44 of \cite{1}, each almost order bounded subset in order continuous Banach lattice is relatively weakly compact.  $A\subseteq L_1(\mu)$ is relatively weakly compact if and only if it is almost order bounded (see\cite{4b}).\\
A vector subspace $G$ of an ordered vector space $E$ is majorizing, whenever  there exists some $y\in G$ with $x\leq y$ for each $x\in E$. A sublattice $G$ of a vector lattice $E$ is said to be order dense in $E$ whenever  there exists some $y\in G$ with $0<y\leq x$ for each $0< x\in E$.
 Recall that a Banach lattice $E$ is said to have
the positive Schur property (the dual positive Schur property) if every positive $w$-null sequence in $E$ (positive $w^*$-null sequence in $E^*$) is norm null. A vector lattice is called laterally complete whenever every subset of pairwise disjoint positive vectors has a supremum.

Unless otherwise stated, throughout this paper, $F$ is a vector lattice, $E$ is a sublattice of $F$, and $X$ is a normed vector space.

\section{The  property  $(F)$ in vector lattices} 
In this section, we investigate on the  property  $(F)$ which is defined in \cite{8}. We will try to get new content from this property.

	Let $E$ be regular in $F$ and order continuous in its own right and closed,  and let $(x_\alpha)\subseteq E$ be an order bounded net and $F$ be a Dedekind complete Banach lattice. By Lemma $4.5$ of \cite{4b}, it is clear that $x_\alpha\xrightarrow{Fo}x$ in $E$ if and only if $x_\alpha\xrightarrow{o}x$ in $E$.

\begin{remark}\label{jj}
	\begin{enumerate}
		\item[i)] Let $E^{**}$ be lattice isomorphic by one sublattice of $F$. If $E$ has  the property $(b)$, then $E$ has the  property  $(F)$. And if $E^{**}$ is lattice isomorphic with $F$, then  the properties $(b)$ and  $(F)$ are equivalent.
		\item[ii)] 	If $E$ is a majorizing sublattice of $F$, then $E$ has the  property  $(F)$. Let $A\subseteq E$ be a $F$-order bounded set. Therefore, there exists a $u\in F^+$ that $A\subseteq [-u,u]$. Since $E$ is majorizing sublattice of $F$, therefore, there exists $v \in E^+$ that $u\leq v$ and $-v\leq -u$. So $A\subseteq [-v , v]$. Hence $A$ is order bounded in $E$.

\item [iii)] Let $E$ be a sublattice of $F$ and $F$ be a sublattice of $G$. If $F$ has the property  $(G)$, necessarily $E$ does not have the property  $(G)$. For example, $c_0$ is a sublattice of $c$ and $c$ is a sublattice of $\ell^\infty$. $c$ has the property $(\ell^\infty)$ while $c_0$ has not property  $(\ell^\infty)$. 
\item [iv)] Let $E$ be a sublattice of $F$ and $F$ be a sublattice of $G$. It is clear that if $E$ has the property  $(G)$, then it has the  property  $(F)$. If $A\subseteq E$ and it is order bounded in $F$, so it is order bounded in $G$. By assumption $A$ is order bounded in $E$.
\end{enumerate}
	\end{remark}

 \begin{theorem}\label{kk}
 	Let $E$ and $F$ be two Banach lattices, and let $E$ has the  property  $(F)$. Then, by one of the following assertions, if $A\subseteq E$ is an almost order bounded in $F$, $A$ is an almost order bounded in $E$.
 	\begin{enumerate}
 		\item[i)] $F$ has an order unit.
 		\item[ii)] $E$ is an ideal of $F$.
 	\end{enumerate}
 \end{theorem}
\begin{proof}
	
		Let $F$ has an order unit and $A\subseteq E$ be an almost order bounded set in $F$. By Theorem 4.21 of \cite{1}, $A$ is an order bounded set in $F$. By the assumption, $A$ is an order bounded set in $E$ and therefore, it is an almost order bounded set in $E$.
		
Suppose that $E$ is an ideal of $F$. Let $A\subseteq E$ be an almost order bounded set in $F$. It means that for each $\varepsilon > 0$, there exists a $u\in F^+$ that $A\subseteq [-u,u] + \varepsilon B_F$. For each $x\in A$, we have $x = x_1 + x_2$ that $x_1 \in [-u,u]$ and $x_2 \in \varepsilon B_F$. We assume that $x\neq 0$. It is obvious that $\mid x_2 \mid\leq \mid \varepsilon \frac{x}{\|x\|}\mid$.  Since $E$ is an ideal of $F$, therefore, $x_2 \in E$. Because $\|x_2\| \leq \varepsilon$, therefore, $x_2 \in \varepsilon B_E$. On the other hand, we have $x_1 = x- x_2$. So $x_1 \in E$.  Since $E$ has the  property  $(F)$, there exists a $v\in E^+$ that $x_1 \in [-v,v]$.  Therefore, $A\subseteq [-v,v] + \varepsilon B_E$. Hence $A$ is almost order bounded in $E$, as regard.
	
	\end{proof}
	
		


\begin{corollary}
	\begin{enumerate}
\item[i)] Let $F$ has an order unit, $E$ has the  property  $(F)$ with order continuous norm and let
  $(x_n)\subseteq E$ be disjoint and almost order bounded sequence in F.  Then $x_n\xrightarrow{\|.\|}0$ in $E$.

	\item[ii)]	Let $(x_n)\subseteq E$ be a disjoint and almost order bounded sequence in $F$, and $E$ be an ideal of $F$ and has the  property  $(F)$. If $E$ has order continuous norm, then $x_n\xrightarrow{\|.\|}0$ in $E$.	
	\end{enumerate}
	\end{corollary}
\begin{proof}
\begin{enumerate}
\item[i)]
By proof of Theorem 1,
 $(x_n)$  is order bounded in $E$ and so by Theorem 4.14 of \cite{1}, $x_n\xrightarrow{\|.\|}0$ in $E$.

\item[ii)]
Since $(x_n)$ is a disjoint sequence, therefore, by Corollary 3.6 of \cite{5g}, we have $x_n\xrightarrow{uo}0$  in $E$. By Theorem \ref{kk}, $(x_n)$ is almost order bounded in $E$. By Proposition 3.7 of \cite{4b}, $x_n\xrightarrow{\|.\|}0$ in $E$.

\end{enumerate}
\end{proof}

\section{$\tilde{o}$rder-norm continuous operators}

	A continuous operator $T: E \rightarrow X $ is said to be
 	$\tilde{o}$rder-norm continuous (or, $\tilde{o}n$-continuous for short),
 	if $(x_\alpha)\subseteq E$ is $\tilde{o}$-null in $ E $, then
 	$ (Tx_\alpha) $ in $X$ is convergent to $0$ in norm. 
  A continuous operator $T: E \rightarrow X $ is said to be
 $\sigma$-$\tilde{o}$rder-norm continuous (or, $\sigma$-$\tilde{o}n$-continuous for short),
 	if $(x_n)\subseteq E$ is $\tilde{o}$-null in $ E $, then
 	$ (Tx_n) $ in $X$ is convergent to $0$ in norm. 
 	
 The collection of all $\tilde{o}n$-continuous operators
 from a vector lattice $E$ into a Banach space $X$ (resp, $\sigma$-$\tilde{o}n$-continuous) will be denoted by $L_{\tilde{o}n}(E,X)$ (resp, $L_{\tilde{o}n}^\sigma(E,X)$.
 
 It is clear that, if $T:E\to X$ is a $\tilde{o}n$-continuous, then $T$ is an order-to-norm topology continuous. By the following example, the converse is not true in general.
\begin{example}\label{gg11}
 	The identity operator $I:c_0 \to c_0$ is order-to-norm topology continuous. Let $(x_\alpha)\subseteq c_0$ is order-null. Since $c_0$ has order continuous norm, therefore, $x_\alpha\xrightarrow{\|.\|}0$.  Consider $(e_n)\subseteq c_0$. $e_n\xrightarrow{\ell^\infty o}0$ in $c_0$. But $(I(e_n))$ is not convergent to zero in norm in $c_0$. Hence $I:c_0\to c_0$ is not $\ell^\infty on$-continuous.
 	\end{example}
 
 Obviously $L_{\tilde{o}n}(E,X)$ is a subspace of $L_{on}(E,X)$.
 In the following, we have some examples from $\tilde{o}n$-continuous operators.
  
 \begin{example}
 \begin{enumerate}
 	\item[i)] If $E$ has  the  property  $(F)$, $E^*$ has order continuous norm and $X$ has the Schur property, then each continuous operator from $E$ to $X$ is a $\sigma$-$\tilde{o}n$-continuous operator. Let $(x_n)\subseteq E$ be $\tilde{o}$-null sequence. Therefore, $(x_n)$ is order null in $F$ and so it is order bounded in $F$. Since $E$ has the property  $(F)$, hence $(x_n)$ is order bounded in $E$. On the other hand $(x_n)$ is $uo$-null in $F$ and so it is $uo$-null in $E$. Because $E^*$ has order continuous norm, therefore, by Theorem 6.4 of \cite{55}, $x_n\xrightarrow{w}0$ in $E$. By continuity of $T$, we have $Tx_n\xrightarrow{w}0$ in $X$. Since $X$ has the Schur property, hence $Tx_n\xrightarrow{\|.\|}0$ in $X$.\\
 	The Banach lattice	$c$ has the property  $(\ell^\infty)$, $c^*$ has order continuous norm and $\ell^1$ has the Schur property, therefore, each continuous operator $T:c\to \ell^1$ is $\sigma$-$\ell^\infty on$-continuous.
 	\item[ii)] Let $F$ be a Dedekind complete Banach lattice. If $E$ has  the  property  $(F)$ and order continuous norm, then each continuous operator from $E$ to $X$ is a $\tilde{o}n$-continuous operator. Let $(x_\alpha)\subseteq E$ is $\tilde{o}$-null net. Therefore, $(x_\alpha)$ is order null in $F$ and so it is order bounded in $F$. Since $E$ has  property  $(F)$erty, hence $(x_\alpha)$ is order bounded in $E$. On the other hand $(x_\alpha)$ is $uo$-null in $F$ and so  by Lemma $4.5$ of \cite{4b}, it is $uo$-null in $E$. Since $(x_\alpha)$ is order bounded, hence $(x_\alpha)$ is order-null in $E$. Because $E$ has order continuous norm, we have $x_\alpha\xrightarrow{\|.\|}0$ in $E$. So $Tx_\alpha\xrightarrow{\|.\|}0$ in $X$.
 	\item[iii)] 	If $T:F\to X$ is a $uon$-continuous operators, then $T|_E : E\to X$ is a $\tilde{o}n$-continuous operator. 	Let $(x_\alpha)\subseteq E$ be a $\tilde{o}$-null net. It is clear that $x_\alpha\xrightarrow{uo}0$ in $F$. By assumption, $Tx_\alpha\xrightarrow{\|.\|}0$ in $X$. 
 \end{enumerate}	
 	\end{example}



 
 The class of $\tilde{o}n$-continuous operators differs from the class of order continuous operators. Since the identity operator $I:c_0\to c_0$ is order continuous while it is not $\ell^\infty on$-continuous (see Example \ref{gg11}.).

\begin{proposition}
	\begin{enumerate}
		\item[i)] Let $T\in L_{\tilde{o}n}(E,X)$, $S:E\to X$ be a continuous operator, and $0\leq S\leq T$, then $S$ is a $\tilde{o}n$-continuous operator.	
	\item[ii)]	Moreover, if $X$ also is a sublattice of $F$, 
		 $T\in L_{\tilde{o}n}(E,X)$ and $S\in L_{\tilde{o}n}(X,Y)$, then $S\circ T \in L_{\tilde{o}n}(E,Y)$.
	\end{enumerate}
	\end{proposition}
\begin{proof}
	\begin{enumerate}
		\item[i)] Let $(x_\alpha)\subseteq E$ be a $\tilde{o}$-null net. It is obvious that $|x_\alpha|\xrightarrow{Fo}0$ in $E$. We have $|Sx_\alpha| \leq |S||x_\alpha| = S|x_\alpha| \leq T|x_\alpha|$. By assumption, $T|x_\alpha|\xrightarrow{\|.\|}0$. So $|Sx_\alpha|\xrightarrow{\|.\|}0$. It means that $S$ is a $\tilde{o}n$-continuous operator.
		\item[ii)] Let $(x_\alpha)\subseteq E$ and $x_\alpha\xrightarrow{Fo}0$. By assumption, we have $Tx_\alpha\xrightarrow{\|.\|}0$. Because each $\tilde{o}n$-continuous operator is continuous, therefore, $STx_\alpha\xrightarrow{\|.\|}0$. Hence $S\circ T \in L_{\tilde{o}n}(E,Y)$.
	\end{enumerate}
	\end{proof}
\begin{remark}
Let $T:E\to G$ be an order continuous lattice homomorphism from  a Dedekind complete vector lattice $E$ to an Archimedean laterally complete 	normed vector lattice $G$. If $E$ is an order dense in the Archimedean vector lattice $F$, then by Theorem 2.32 of \cite{1}, $T$ exteneded from $F$ to $G$ that is an order continuous lattice homomorphism. Now if $G$ has order continuous norm, then $T$ is a $\tilde{o}n$-continuous.
	\end{remark}

\begin{theorem}\label{pp}
For an order bounded operator $T:E\to G$ between two Riesz spaces with $G$ Dedekind complete and $G$ has order continuous norm Banach lattice, the following statements are equivalent.  
	\end{theorem}
\begin{enumerate}
	\item[(1)] $T$ is $\tilde{o}n$-continuous. 
	\item[(2)] If  $ (x_{\alpha})\subseteq E$ and $ x_{\alpha} \downarrow 0 $ holds in $F$, then $ Tx_{\alpha}$ is norm convergent to $0$ in $G$. 
	\item[(3)] If $ (x_{\alpha})\subseteq E$ and $ x_{\alpha} \downarrow 0 $ holds in $F$, then $\inf \left\lbrace    \Vert    Tx_{\alpha}   \Vert   \right\rbrace \, =\, 0$.
    \item[(4)] $T^{+}$, $T^{-}$ are both  $\tilde{o}n$-continuous.
    \item[(5)] $|T|$ is $\tilde{o}n$-continuous.
\end{enumerate}
\begin{proof}
$(1)\Rightarrow (2)$ and $(2) \Rightarrow (3)$ are obvious. Clearly, if $T$ is a positive operator then (1), (2), and (3) are equivalent.\\
 	$(3)\Rightarrow (4)$ It is enough to show that $T^{+}$  is $\tilde{o}n$-continuous. To this end, let $( x_{\alpha})\subseteq E$ and $ x_{\alpha} \downarrow 0 $ in $F$. Let $\Vert T^{+} x_{\alpha}  \Vert  \downarrow z\, \geq 0$. We have to show that $z\,=\,0$. Fix some index $\beta$ and put $x\,=\,x_{\beta}$.\\
 Now for each $0\leq y\leq x$ and each $\alpha\succeq \beta$ we have
\begin{equation}
0 \,\leq \, y - y\wedge x_{\alpha} = y\wedge x - y\wedge x_{\alpha} \,\leq \, x - x_{\alpha},
\end{equation}
and consequently
\begin{equation}
Ty - T(y\,\wedge \, x_{\alpha})\, = T(y - y\,\wedge \, x_{\alpha})\leq T^{+}( x - x_{\alpha}) =T^{+} x - T^{+} x_{\alpha},
\end{equation}
from which it follows that
\begin{equation}
0\,\leq \, z \,\leq \, T^{+} x_{\alpha} \,\leq \, T^{+} x + \vert T(y\, \wedge \, x_{\alpha}) \vert - Ty 
\end{equation}
holds for all $\alpha\succeq \beta$ and all $0\leq y\leq x$. Now since for each fixed vector 
\\ $0\leq y\leq x$ we have $ y \, \wedge \, x_{\alpha} \downarrow _{\alpha \,\succeq \beta} 0 $ in $F$, it then follows from our hypothesis that 
$\inf \left\lbrace    \Vert    T(y\,\wedge \,x_{\alpha})   \Vert   \right\rbrace \, =\, 0$, and hence from (3) we see that 
$0\,\leq \, \Vert  z  \Vert \,\leq \, \Vert T^{+} x - Ty \Vert $
holds for all $0\leq y\leq x$. In view of 
$ T^{+} x \, = \, \sup \left\lbrace Ty :\,\, 0\,\leq \, y \, \leq \, x \right\rbrace $ and by order continuity of $G$,
the latter inequality yields $z\,=\,0$, as desired.
\\$(4)\Rightarrow (5)$ The implication follows from the identity $\vert T \vert \,=\, T^{+} + T^{-}$.
\\$(5)\Rightarrow (1)$ The implication follows easily from the lattice inequality 
$\vert Tx \vert \,\leq \, \vert T\vert \vert x \vert $.
\end{proof}

\begin{theorem}
If $E$ and $G$ two Riesz spaces with $G$ Dedekind complete and $G$ has order continuous norm Banach lattice, the following assertions are true.
\end{theorem}
\begin{enumerate}
	\item[i)] The set of all order bounded $\tilde{o}n$-continuous operators from $E$ into $G$ is a band of 
$L_b(E,G)$. 
     \item[ii)] The set of all order bounded $\sigma$-$\tilde{o}n$-continuous operators from $E$ into $G$ is a band of 
$L_b(E,G)$. 
\end{enumerate}
\begin{proof}
\begin{enumerate}
	\item[i)] Let $T:E\to G$ be an order bounded $\tilde{o}n$-continuous operator. Note that if $\vert S\vert \, \leq \, \vert T \vert$ holds in $L_b(E,G)$, then by Theorem 2,  $S$ is order bounded  and $\tilde{o}n$-continuous operator. That is, the set of all order bounded $\tilde{o}n$-continuous operators from $E$ into $G$ is an ideal of $L_b(E,G)$. 
Let $0\,\leq \, T_{\lambda}\uparrow T$ in $L_b(E,G)$ that $T_{\lambda}$ is an $\tilde{o}n$-continuous operator for all $\lambda $, and let $( x_{\alpha})\subseteq E$ that $0\,\leq \, x_{\alpha}\uparrow x$ in $F$. Then for each fixed index $ \lambda$ we have 
\begin{equation}
0\,\leq \, T (x - x_{\alpha}) \,\leq \, (T-T_{\lambda})(x) + T_{\lambda}(x - x_{\alpha}),
\end{equation}
and $x - x_{\alpha}\downarrow 0$, in conjunction with $T_{\lambda}$ is an $\tilde{o}n$-continuous operator, implies
\begin{equation}
0\,\leq \, \inf_{\alpha} \left\lbrace  \Vert T(x - x_{\alpha}) \Vert \right\rbrace \, \leq \, \Vert (T - T_{\lambda})(x)\Vert
\end{equation}
for all $\lambda$. From $T - T_{\lambda}\downarrow 0$ and by order continuity of $G$, we see that  
$\inf _{\alpha}\left\lbrace  \Vert T(x - x_{\alpha}) \Vert \right\rbrace \, =\,0$. Thus, $T$ is $\tilde{o}n$-continuous operator, and the proof follows. 
\item[ii)] The proof is similar to (i).
\end{enumerate}
\end{proof}

\begin{theorem}\label{pp}
Let $T:E\to G$ be an order bounded operator. Then the following assertions are true.  
	\end{theorem}
\begin{enumerate}
	\item[i)] If $G$ is Archimedean vector lattice and $T$ is a preserves disjointness and  $\tilde{o}n$-continuous, then $|T|$ exists and $|T|\in L_{\tilde{o}n}(E,G)$.
	\item[ii)] If $E$ is a band in $F$, $G$ is an atomic with order continuous norm Banach lattice and $T:E\to G$ is $\sigma$-$\tilde{o}n$-continuous, then $|T|$ exists and $|T|\in L_{\tilde{o}n}^\sigma(E,G)$.
\end{enumerate}
\begin{proof}
\begin{enumerate}
	\item[i)] Let $(x_\alpha)\subseteq E$ be a $\tilde{o}$-null net. By assumption we have $Tx_\alpha\xrightarrow{\|.\|}0$. By Theorem 2.40 of \cite{1}, $|T|$ exists and $|T||x| = |T|x|| = |Tx|$  for all $x\in E$. Since, $||T|x_\alpha| \leq |T||x_\alpha| \xrightarrow{\|.\|}0$, therefore, $|T|x_\alpha\xrightarrow{\|.\|}0$. Hence $|T|\in L_{\tilde{o}n}(E,G)$. 
	\item[ii)] Let $(x_n)\subseteq E$ and $x_n\xrightarrow{o}0$ in $E$. It is clear that $x_n\xrightarrow{Fo}0$ in $E$. By assumption, $Tx_n\xrightarrow{\|.\|}0$ in $G$. It is obvious that, $(x_n)$ is order bounded and therefore, $(Tx_n)$ is order bounded. Hence by Lemma 5.1 of \cite{55}, $Tx_n\xrightarrow{o}0$ in $G$. Hence $T$ is an $\sigma$-order continuous operator. Note that since $G$ has order continuous norm, therefore, it is a Dedekind complete. So by Theorem  1.56 of \cite{1}, $|T|$ exists and it is an $\sigma$-order continuous. Let $(x_n)\subseteq E$ be a $\tilde{o}$-null net. We have $|x_n| = P_E |x_n|\leq P_E (y_m)$. We have $x_n\xrightarrow{o}0$ in $E$. By assumption, $|T|x_n\xrightarrow{o}0$ in $G$. Because $G$ has order continuous norm, therefore, $|T|x_n\xrightarrow{\|.\|}0$ in $G$. Hence $|T|\in L_{\tilde{o}n}^\sigma(E,G)$
\end{enumerate}
	\end{proof}
\begin{corollary}
	By part 2 of Theorem \ref{pp},
if $E$ is a band in $F$, $G$ is an atomic with order continuous norm Banach lattice. It follows that $T:E\to G$ is a $\sigma$-$\tilde{o}n$-continuous operator if and only if it is a $\sigma$-order continuous operator. Therefore, by Theorem 1.57 of \cite{1}, $L_{\tilde{o}n}^\sigma(E,G)$ is a band of $L_b(E,G)$. 
	\end{corollary}


 \section{$\tilde{o}$rder weakly compact operator}
 
 A continuous operator $T: E \rightarrow X $ is said to be
 $\tilde{o}$rder weakly compact (or, $\tilde{o}$-weakly compact for short),
 if $A\subseteq E$ is $Fo$-bounded in $ E $, then
 $ T(A) $ in $X$ is a relatively weakly compact set. 
The collection of all $\tilde{o}$-weakly compact operators
from vector lattice $E$ into Banach space $X$ will be denoted by $W_{\tilde{o}}(E,X)$.

 A subset $A$ in a Banach lattice $E$ is  $F$-almost order bounded if for any $\epsilon>0$ there exists $u\in F^+$ such that $A\subseteq [-u,u]+\epsilon B_E$.

 As following remark,  very weakly compact operator  $T: E \rightarrow X $ is a $\tilde{o}$-weakly compact operator, the converse holds whenever  $E$ has an order unit.

\begin{remark}
\begin{enumerate}
\item[i)]
Let  $T: E \rightarrow X $ be an weakly compact operator.  If $A$ is a $F$-order bounded set in $E$, so it is norm bounded set. Since $T$ is a weakly compact operator, $T(A)$ is relatively weakly compact set in $X$. It means that $T$ is a $\tilde{o}$-weakly compact operator.

\item[ii)]
 Let $E$ has an order unit and $T:E\to X$ is a $\tilde{o}$-weakly compact operator.  If $A$ is a norm bounded set in $E$,  then it is order bounded set in $E$ and therefore, it is $F$-order bounded. By assumption, $T(A)$ is relatively weakly compact set in $X$. It means that $T$ is a weakly compact operator.
	\end{enumerate}
		\end{remark}
\begin{proposition}
	 If $E$ has order continuous norm with   property  $(F)$, then the identity operator $I:E\to E$ is $\tilde{o}$-weakly compact
	\end{proposition}
\begin{proof}
	Let $A\subseteq E$ be a $Fo$-bounded set. Since $E$ has the  property  $(F)$, therefore, $E$ is an order bounded set in $E$. It is obvious that $A$ is almost order bounded in $E$. Since $E$ has order continuous norm,  $A$ is a relatively weakly compact set in $E$ by Theorem 4.9(5) and Theorem 3.44 of \cite{1}. Hence $I(A)$ is a relatively weakly compact set in $E$. It means that $I$ is a $\tilde{o}$-weakly compact operator.
	\end{proof}

\begin{theorem}\label{uytt}
An operator $T:E\to X$  is $\tilde{o}$-weakly compact if and only if for each disjoint and $Fo$-bounded sequence $(x_n)\subseteq E$ implies $Tx_n\xrightarrow{\|.\|}0$.
\end{theorem}
\begin{proof}
	Let  $T:E\to X$ be  a $\tilde{o}$-weakly compact operator. Therefore, for each $u\in F^+$, $T([-u,u])$ is relatively weakly compact. Let $I_u$ be the ideal generated by $u$ in $E$. The operator $T|_{I_u}:I_u \to X$ is weakly compact operator. Since $I_u$ is an $AM$-space with order unit, therefore, $T\arrowvert_{I_u}:I_u\to F$ is $M$-weakly. Hence for each disjoint norm bounded sequence $(x_n)\subseteq I_u$,  we have $Tx_n\xrightarrow{\|.\|}0$. 

Conversely, let $A\subseteq E$ be a $Fo$-bounded set. Then  there exists $u\in F^+$ such that $A\subseteq [-u,u]$. Let $I_u$ be the ideal generated by $u$ in $E$ and $(x_n)\subseteq A$ be a disjoint sequence. It is clear that $(x_n)$ is norm bounded. By assumption, we have $Tx_n\xrightarrow{\|.\|}0$ in $F$. Therefore, $T:I_u  \to X$ is $M$-weakly compact and  by Theorem 5.61 of \cite{1}, $T:I_u  \to X$  is a weakly compact operator. Let $A\subseteq E$ be a $Fo$-bounded set. Then there exists $u\in F^+$ that $A$ is norm bounded in $I_u$ and $T:I_u\to X$ is weakly compact. Therefore, $T(A)$ is a relatively weakly compact in $X$. So $T:E\to X$ is a $\tilde{o}$-weakly compact operator.
\end{proof}

\begin{corollary}
\begin{enumerate}
	\item[i)] 	Let $T, S$ be two operators that $0\leq T\leq S$ and $S$ is a $\tilde{o}$-weakly compact operator. If $(x_n)\subseteq E$ is disjoint and $Fo$-bounded, then by Theorem \ref{uytt}, $ Sx_n\xrightarrow{\|.\|}0$. It follows that $Tx_n\xrightarrow{\|.\|}0$. So $T$ is a $\tilde{o}$-weakly compact operator.
	\item[ii)]  Let $T, S$ be two $\tilde{o}$-weakly compact operators. By Theorem \ref{uytt}, it is clear that $S\circ T$ is a $\tilde{o}$-weakly compact operator.
\end{enumerate}
\end{corollary}



It is obvious that, if $T:E\to X$ is a $\tilde{o}$-weakly compact operator, then is order weakly compact. By following example the converse is not true in general.

\begin{example}
		The operator $T:\ell^{1} \rightarrow \ell^{\infty} $ defined by 
	\[
	T(x_{1},x_{2},\ldots) = (\sum_{i=1}^\infty x_i , \sum_{i=1}^\infty x_{i},\ldots)
	\]
	is an order weakly compact operator. Let $(x_n)\subseteq \ell^1$ be disjoint order bounded sequence. We have $x_n\xrightarrow{uo}0$ and order bounded. Therefore, $x_n\xrightarrow{o}0$. Since $\ell^1$ has order continuous norm, therefore, $(x_n)$ is norm-null. Because $T$ is a continuous operator, hence $Tx_n\xrightarrow{\|.\|}0$ in $\ell^\infty$. So by Theorem 5.57 of \cite{1}, $T$ is an order weakly compact operator.\\ 
	If we consider $(e_n)\subseteq \ell^1$ we have $e_n \xrightarrow{\ell^\infty o}0$ in $\ell^1$. On the other hand, $Te_{n}=(1,1,1,\ldots)$, 
	therefore, $ (Te_n) $ is not convergent to zero in norm topology. Thus, $ T $ 
	is not $\tilde{o}$-weakly compact.
	\end{example}

\begin{theorem}
	Let $G$ be a normed vector lattice that is a sublattice of normed vector lattice $H$ and $T:E\to G$ be a $\tilde{o}$-weakly compact operator. By one of the following conditions, the modulus of $T$ exists and it is a $\tilde{o}$-weakly compact operator.

\begin{enumerate}
\item[i)] $E$ is an $AL$-space and $X$ has the property $(P)$ and the $H$-property.
\item[ii)]  $E$ and $G$ have an order unit.
\item[iii)] $G$ is Archimedean Dedekind complete and $ T $ is an order bounded preserves disjointness.
\end{enumerate}
\end{theorem}
\begin{proof}
	\begin{enumerate}
	\item[i)] Let $(x_n)\subseteq E$ be a disjoint order bounded sequence. It is obvious that $(x_n)$ is  $Fo$-bounded. By assumption and Theorem \ref{uytt}, $Tx_n\xrightarrow{\|.\|}0$. Hence by Theorem 5.57 of \cite{1}, $T$ is an order weakly compact operator. Since $E$ is an $AL$-space and $G$ has the property $(b)$, by Theorem 2.2 of \cite{11}, $|T|$ exists and is an order weakly compact operator. Since $G$ has the $H$-property, $|T|$ is a $\tilde{o}$-weakly compact operator.
	\item[ii)] Let $A$ be a norm bounded set in $E$.	Since $E$ has an order unit, therefore, $A$ is order bounded and so is $Fo$-bounded. By assumption, $T(A)$ is a relatively weakly compact set in $G$. Hence $T$ is a weakly compact operator. Since $G$ has an order unit, therefore, by Theorem 2.3 of \cite{15}, the modulus of $T$ exists and it is a weakly compact operator. It is obvious that $|T|$ is a $\tilde{o}$-weakly compact operator.
\item[iii)] By Theorem 2.40 of \cite{1}, 
$ |T| $ exists and we have $|T| |x| =  |T|x|| = |Tx|$  for all $x$.	
If $(x_n) \subseteq E$ is a $Fo$-bounded disjoint sequence, then by assumption $Tx_n\xrightarrow{\|.\|}0$. 
We have
\(
|T||x_{n}|=|T|x_{n}||=|Tx_{n}| \xrightarrow{\|.\|} 0
\)
in $ G$ for each $n$. Now by inequality $ | | T | x_{n} | \leq |T | | x_{n} |$,
we have  $ | T| x_{n} \xrightarrow{\|.\|} 0$. Hence $|T|$ is a $\tilde{o}$-weakly compact operator.
	\end{enumerate}
	\end{proof}



	


The following examples shows that  $\tilde{o}$-weakly compact operators  do not have the duality property.
\begin{example}
	\begin{enumerate}
\item[i)]  Consider the operator $T:C[0,1] \rightarrow c_0$, given by 
$$ T(f) = (\int_0 ^1 f(x)\sin x dx, \int_0 ^1 f(x)\sin2x dx,...).$$ By Example 3.15  of \cite{44}, $T$ is a $wun$-Dunford-Pettis and by Theorem  3.11 of \cite{44}, $T$ is a weakly compact operator. Therefore, $T$ is a $\tilde{o}$-weakly compact operator. 
We have $T^* :\ell^1\to (C[0,1])^*$ that 
$$T^*x_n(f)=\sum_{n=1}^\infty x_n (\int_0^1 f(t)sinntdt).$$
Note that $(e_n)\subseteq \ell^1$ is $\ell^\infty$-order bounded and disjoint. Put $f_n(t) = sinnt$ for all $n$. We have 
$$\| T^*e_n\| \geq \|T^*e_n(f_n)\| = \int_0^1 (sinnt)^2 dt\nrightarrow 0.$$
 Thus by Theorem \ref{uytt}, $T^*$ is not a $\tilde{o}$-weakly compact operator.
\item[ii)]  Consider the functional $f:\ell^1 \rightarrow \mathbb{R}$ defined by 
$$f(x_1,x_2,...)= \sum_1^\infty x_i.$$
  $(e_n)\subseteq \ell^1$ is $\ell^\infty$-order bounded and disjoint while $f(e_n)\nrightarrow 0$. Therefore, by Theorem \ref{uytt}, $f$ is not a $\tilde{o}$-weakly compact operator, but 
  it is obvious that $f^*:\mathbb{R}\to \ell^\infty$ is a $\tilde{o}$-weakly compact operator.
\end{enumerate}
\end{example}
In the following, under some conditions, we show that  if an operator $T$ is $\tilde{o}$-weakly compact, then its adjoint $T^*$ is also $\tilde{o}$-weakly compact and vice versa.
\begin{proposition}
Let $G$ be a vector lattice such that $G^*\subseteq F$. Then the following assertions are true.
	\begin{enumerate}
		\item[i)] If $E$ has an order unit and $T:E\to G$ is $\tilde{o}$-weakly compact, then $T^*$ is $\tilde{o}$-weakly compact.
		\item[ii)] If $G^*$ has an order unit and  $T^*:G^*\to E^*$ is $\tilde{o}$-weakly compact, then $T$ is $\tilde{o}$-weakly compact.
		\end{enumerate}
\end{proposition}
\begin{proof}
	\begin{enumerate}
		\item[i)] Let $E$ has an order unit and $T:E\to G$ is a $\tilde{o}$-weakly compact operator. It is obvious that $T$ is a weakly compact operator. By Theorem    5.23 of \cite{1}, $T^*$ is a weakly compact operator and therefore, $T^*$  is a $\tilde{o}$-weakly compact operator. \label{gf}
	\item[ii)] The proof is similar to (\ref{gf}).
	\end{enumerate}
\end{proof}


\begin{theorem}
	Let $T:F\to X$ be an operator. $T\mid_{E}:E\to X$ is $\tilde{o}$-weakly compact if and only if   $T(A)$ is relatively weakly compact for each $F$-almost order bounded set $A\subseteq E$.
	\end{theorem}
\begin{proof}
	If $T(A)$ is relatively weakly compact  for each $F$-almost order bounded subset $A$ of $E$, it is obvious that $T\mid_E$ is $\tilde{o}$-weakly compact.
	
	Conversely, let $A\subseteq E$ be a $F$-almost order bounded. Therefore, there exists a $u\in F^+$ that $A\subseteq [-u,u] +\varepsilon B_E$. It is obvious that $T(A) \subseteq T[-u,u] + \varepsilon T(B_E)$. Since $T$ is $\tilde{o}$-weakly compact, hence $T[-u,u]$ is a relatively weakly compact. By Theorem 3.44 of \cite{1}, $T(A)$ is a relatively weakly compact  in $X$.
	\end{proof}










\begin{thebibliography}{1}
\providecommand{\url}[1]{{#1}}
\providecommand{\urlprefix}{URL }
\expandafter\ifx\csname urlstyle\endcsname\relax
  \providecommand{\doi}[1]{DOI~\discretionary{}{}{}#1}\else
  \providecommand{\doi}{DOI~\discretionary{}{}{}\begingroup
  \urlstyle{rm}\Url}\fi






\bibitem{1} {C. D. Aliprantis and O. Burkinshaw,}
  {Positive operators.} vol. 119. Springer Science, Business Media (2006).






$•$\bibitem{65} {{\c{S}}. Alpay and B. Altinand C. Tonyali,} 
{ On property (b) of vector lattice.} \newblock Positivity  \textbf{7} (1)  (2003) 135-139.
  
  \bibitem{11} {B. Aqzzouz  and J. H'michane,}
  {Some results on order weakly compact operators}. Mathematica Bohemica,  {\bf 134} (4)  (2009)  359--367.

\bibitem{8}{K. Haghnejad Azar},
{A generalization of order convergence in the vector lattices}, Facta Universitis  (NIˇS)
Ser. Math. Inform.  {\bf 37} (3 ) (2022)  521–528.

\bibitem{55}  {Y. Deng and M. O’Brien and V. G. Troitsky,}
{Unbounded norm convergence in Banach lattices}. Positivity. {\bf 21}  (2017)  963--974. 

\bibitem{5g} { N. Gao and V.g. Troitsky and F. Xanthos},
{uo-Convergence and its applications to Ces\`aro means in Banach lattices}. \newblock Isr. J. Math. \textbf{ 220}   (2017)  649–689.


\bibitem{4b} { N. Gao and F. Xanthos},
{Unbounded order convergence and application to martingales without probability}. J. Math. Anal. Appl. {\bf 415}  (2014)  931--947. 


\bibitem{p1} {S. A. Jalili,  K. Haghnejad Azar and M. B. Farshbaf Moghimi,}
 {Order-to-topology continuous operators}, Positivity  {\bf 25}  (2021)  1313–1322.


\bibitem{12o} { K.  Haghnejad Azar, M.  Matin  and R.  Alavizadeh,}
{Unbounded order-norm continuous and  unbounded norm continuous operators}, Filomat {\bf 35} (13)  (2021)  4417–4426.


\bibitem{44} { K.  Haghnejad Azar, M.  Matin  and R.  Alavizadeh,}
{Weakly Unbounded Norm Topology and $wun$-Dunford-Pettis Operators}, To appear in
Rendiconti del Circolo Matematico di Palermo Series 2, (2022). 


\bibitem{15} {K. D. Schmidt,}
 {On the modulus of weakly compact operators and strongly additive vector measures},  Proc. Amer. Math. Soc. \bf{ 102} (4)   (1988) 862-866.




\end{thebibliography}
\end{document}